\documentclass[reqno,12pt]{amsart}
\usepackage[a4paper, margin=1in]{geometry}
\usepackage{hyperref}
\usepackage[T1]{fontenc}
\makeatletter
\@namedef{subjclassname@2020}{%
  \textup{2020} Mathematics Subject Classification}
\makeatother
\title[Hilbert property for degree-one del Pezzo surfaces]{Hilbert property for low-genus families and degree-one del Pezzo surfaces}
\subjclass[2020]{14G05, 14J20 (14G25, 14G27, 11G35).}
\author{Julian Demeio}
\address{Julian Demeio \\ 
Institut f\"ur Algebra, Zahlentheorie und Diskrete Mathematik \\
Leibniz Universit\"at Hannover \\
Welfengarten 1  \\
Hannover \\
30167 \\
Germany
}
\email{demeiojulian@yahoo.it}
\author{Sam Streeter}
\address{Sam Streeter \\
School of Mathematics \\
University of Bristol \\
Woodland Road \\
Bristol \\
BS8 1UG \\
UK}
\email{sam.streeter@bristol.ac.uk}
\author{Rosa Winter}
\address{Rosa Winter \\
Mathematisches Institut \\
Albert-Ludwigs-Universit\"at Freiburg \\
Ernst-Zermelo-Stra{\ss}e 1 \\
79104 Freiburg \\
Germany}
\email{rosa.winter@math.uni-freiburg.de}
\usepackage{xcolor}
\usepackage{amsmath}
\usepackage{amssymb}
\usepackage{amsthm}
\usepackage{enumitem}
\usepackage{hyperref}
\usepackage{tikz-cd}
\usepackage{color,soul}
\usepackage{mathrsfs}
\theoremstyle{definition}
\newtheorem{mydef}{Definition}[section]
\newtheorem{note}[mydef]{Note}
\newtheorem{remark}[mydef]{Remark}
\newtheorem*{question}{Question}

\newtheorem{example}[mydef]{Example}
\theoremstyle{plain}
\newtheorem{theorem}[mydef]{Theorem}
\newtheorem{proposition}[mydef]{Proposition}
\newtheorem{lemma}[mydef]{Lemma}

\let\originalleft\left
\let\originalright\right
\renewcommand{\left}{\mathopen{}\mathclose\bgroup\originalleft}
\renewcommand{\right}{\aftergroup\egroup\originalright}

\renewcommand{\P}{\mathbb{P}}

\newcommand{\A}{\mathbb{A}}

\DeclareMathOperator{\ch}{char}

\DeclareMathOperator{\Gal}{Gal}

\DeclareMathOperator{\Pic}{Pic}

\DeclareMathOperator{\Spec}{Spec}

    \DeclareFontFamily{U}{wncy}{}
    \DeclareFontShape{U}{wncy}{m}{n}{<->wncyr10}{}
    \DeclareSymbolFont{mcy}{U}{wncy}{m}{n}
    \DeclareMathSymbol{\Sh}{\mathord}{mcy}{"58} 
\makeatletter
\newcommand*{\da@rightarrow}{\mathchar"0\hexnumber@\symAMSa 4B }
\newcommand*{\da@leftarrow}{\mathchar"0\hexnumber@\symAMSa 4C }
\newcommand*{\xdashrightarrow}[2][]{%
  \mathrel{%
    \mathpalette{\da@xarrow{#1}{#2}{}\da@rightarrow{\,}{}}{}%
  }%
}
\newcommand{\xdashleftarrow}[2][]{%
  \mathrel{%
    \mathpalette{\da@xarrow{#1}{#2}\da@leftarrow{}{}{\,}}{}%
  }%
}
\newcommand*{\da@xarrow}[7]{%
\sbox0{$\ifx#7\scriptstyle\scriptscriptstyle\else\scriptstyle\fi#5#1#6\m@th$}%
 \sbox2{$\ifx#7\scriptstyle\scriptscriptstyle\else\scriptstyle\fi#5#2#6\m@th$}%
  \sbox4{$#7\dabar@\m@th$}%
  \dimen@=\wd0 %
  \ifdim\wd2 >\dimen@
    \dimen@=\wd2 %
  \fi
  \count@=2 %
  \def\da@bars{\dabar@\dabar@}%
  \@whiledim\count@\wd4<\dimen@\do{%
    \advance\count@\@ne
    \expandafter\def\expandafter\da@bars\expandafter{%
      \da@bars
      \dabar@ 
    }%
  }%
  \mathrel{#3}%
  \mathrel{%
    \mathop{\da@bars}\limits
    \ifx\\#1\\%
    \else
      _{\copy0}%
    \fi
    \ifx\\#2\\%
    \else
      ^{\copy2}%
    \fi
  }%
  \mathrel{#4}%
}
\makeatother
\begin{document}
\begin{abstract}
We prove that the Hilbert property is satisfied by certain del Pezzo surfaces of degree one and Picard rank one over fields finitely generated over $\mathbb{Q}$. We generalize results of the first author on elliptic surfaces and employ constructions used by Desjardins and the third author to prove density of rational points. Our results are the first on the Hilbert property for minimal del Pezzo surfaces of degree one without a conic fibration.
\end{abstract}
\maketitle
\setcounter{tocdepth}{1}
\tableofcontents
\section{Introduction}
This work is primarily concerned with the following long-standing question of Colliot-Th\'el\`ene and Sansuc \cite{CTS} about the abundance of rational points on unirational varieties.
\begin{question}[Colliot-Th\'el\`ene--Sansuc]
Does every projective unirational variety over a Hilbertian field satisfy the Hilbert property?
\end{question}
Informally, the Hilbert property holds for a variety $X$ defined over a field $K$ if \emph{Hilbert's Irreducibility Theorem} (from which it takes its name) holds for polynomials with coefficients in the function field $K(X)$, and it is a measure of the abundance of rational points of $X$. Hilbert's original theorem essentially establishes the Hilbert property for $\mathbb{A}^n_K$ for $K$ a number field. See Section \ref{sec:HP} for the precise definition. The Hilbert property is closely connected to the \emph{inverse Galois problem}: if all unirational varieties over a number field $K$ satisfy the Hilbert property, then for any finite group $G$, there exists a Galois extension $L/K$ with $\Gal(L/K) \cong G$ \cite[Proof~of~Thm.~3.5.9]{SER}.

A now well-developed technique to prove the Hilbert property for a surface $X$ is to produce an ``abundance'' of curves of geometric genus $0$ or $1$ on $X$ and use these to propagate rational points. When one is able to produce an abundance of curves of genus $0$ (e.g.\ at least enough to infer the unirationality of $X$), one usually hopes to prove more than the Hilbert property: see for instance the paper \cite{SD} by Swinnerton--Dyer and the papers \cite{DS,DSW} by the authors of this paper, where weak weak approximation of the surfaces under consideration is proven. When one is able to produce many curves of genus $1$, but not (necessarily) curves of genus $0$, the Hilbert property is usually considered the best hope with the current technology.

Historically, the first use of curves of genus $1$ in this context seems to be Manin's proof of the Hilbert property of cubic surfaces with a rational point \cite[§46]{M}, by employing the $2$-dimensional linear system of cubic curves on the surface. (Consistently with the philosophy laid out above, the aforementioned result of Swinnerton--Dyer strenghtened the result of Manin by ingeniously using the curves of genus $0$ as well.) After Manin, not much work was done in this direction until the paper of Corvaja and Zannier \cite{CZ} proving the Hilbert property for the K3 surface $x^4+y^4=z^4+w^4$ using two elliptic fibrations. Their result was generalized by the first-named author of this paper to any surface with multiple genus-$1$ fibrations and Zariski-dense rational points, assuming a simple connectedness hypothesis for a certain Zariski-open subset of the surface \cite[Thm.~1.2]{DEM21}. Since then, the {\em multiple fibrations method} has appeared in multiple works in analogous contexts: see for instance \cite{STR} (for conic bundles), \cite{GCM1,GCM2} (for the potential analogue), and \cite{COC0,COC,AL} (for the integral analogue).

In this paper, we generalize \cite[Thm.~1.2]{DEM21} by replacing one of the fibrations with a general family of curves of {\em low genus} (0 or 1). Unlike in Manin’s work and in the multiple fibrations method, our family may be non-linear (as it is in our main application, Theorem \ref{thm:2}). This introduces an additional difficulty: in the linear case, every rational point lies on at least one member of the family defined over $K$, and this is no longer true in the non-linear setting. A similar use of non-linear families to study the Hilbert property first appeared in the work \cite{COC} of Coccia in the analogous context of integral points, but to our knowledge this paper is the first to apply this idea to {\em rational} points and to establish a general result of the following kind.

\begin{theorem} \label{thm:1}
Let $S$ be a smooth projective geometrically connected surface over a field $K$ which is finitely generated over $\mathbb{Q}$. Suppose we are given $n \geq 0$ genus-$1$ fibrations $\pi_i: S \rightarrow \mathbb{P}^1$, $1 \leq i \leq n$ and a low-genus family of curves $\{X_b\}_{b \in B}$ on $S$ with some base $B/K$ such that the induced map $\phi\colon X \to S$ from their total space $X$ is dominant with geometrically integral generic fiber. Let $F \subset S$ denote the union of divisors on $S$ which are vertical with respect to each $\pi_i$ and do not intersect a general member of the family $\{X_b\}_{b \in B}$ transversally. If $S \setminus F$ is algebraically simply connected and $X(K) \subset X$ is Zariski-dense, then $S$ satisfies the Hilbert property over $K$.
\end{theorem}

Recall that a smooth variety $V$ is \emph{algebraically simply connected} if any finite surjective morphism $V' \rightarrow V$ of degree at least $2$ is ramified. The assumption that $\phi$ has geometrically integral generic fiber is necessary for the method we employ to work (see Remark \ref{rem:hypothesis is necessary}). The other hypothesis of Zariski-density on $X$, standing in place of the more common and natural hypothesis of Zariski-density of rational points on $S$, arises from the ``non-linear difficulty'' mentioned above.

We apply Theorem \ref{thm:1} to certain families of \emph{del Pezzo surfaces} (smooth projective varieties $X$ of dimension $2$ with ample anticanonical class $-K_X$). It may be that all del Pezzo surfaces with a rational point are unirational and so within the remit of the above question \cite[Rem.~9.4.11]{POO}. Since del Pezzo surfaces of degree one always possess a rational point (the base-point of the anticanonical linear system), we are thus motivated to prove that they satisfy the Hilbert property. 

The arithmetic complexity of del Pezzo surfaces is notionally governed by the degree $d = (-K_X)^2 \in \{1,\dots,9\}$. Let $X/k$ be a degree-$d$ del Pezzo surface over a Hilbertian field $k$ with a rational point. For $d \geq 5$, such $X$ is rational \cite[Thm.~9.4.29]{POO}; as the Hilbert property is a birational invariant of smooth varieties, Hilbert's original result then implies the Hilbert property. The Hilbert property is established for $d = 4$ in work of the second author \cite{STR}. Previous joint work of the authors \cite{DSW} establishes the Hilbert property when $d = 2$ and $X/k$ contains a rational point outside a closed subset of $X$. This result immediately implies the Hilbert property for $d = 3$ \cite[Cor.~3.7]{DSW}.

When $d = 1$, the first two authors \cite{DS} proved the Hilbert property when $k$ is a number field and $X$ possesses a general conic fibration. Little is known about rational points on minimal degree-one del Pezzo surfaces with no conic fibration, which are those of Picard rank 1. Desjardins and the third author \cite{DW} proved density of rational points for an infinite family of del Pezzo surfaces of degree one. The surfaces in this family are endowed with a family of low-genus curves contained (non-linearly) in $|{-3K_S}|$. It is this family of curves that Desjardins and the third author use to prove density. Nijgh \cite{NIJ} later generalized the work \cite{DW} by enlarging the family of surfaces to which the method can be applied. Salgado and van Luijk \cite{SvL} proved that $\mathbb{Q}$-rational del Pezzo surfaces with dense rational points are dense in a parameter space for real del Pezzo surfaces, following work of V\'arilly-Alvarado \cite{VA}, Ulas \cite{ULA1,ULA2}, Ulas and Togb\'e \cite{UT} and Jabara \cite{JAB}.

The primary result of this paper is the proof of the Hilbert property for the family of surfaces for which Nijgh proves Zariski-density of rational points.

\begin{theorem} \label{thm:2}
Let $K$ be a field finitely generated over $\mathbb{Q}$. Let $a_4(u) = au + b,a_6(u) = cu^2 + du + e \in K[u]$ be polynomials and 
$f(t) := f_3t^3 + f_2 t^2 + f_1 t + f_0 \in K[t]$ be a cubic polynomial. Let $S \subset \mathbb{P}(2,3,1,1)$ be the surface given by
\[
S: y^2  = x^3 + a_4(f(z/w))x w^4 + a_6(f(z/w))w^6.
\]
Assume that $S$ is smooth, hence a del Pezzo surface of degree one. Let $\rho: \mathcal{E} = \text{Bl}_{\mathcal{O}} S \rightarrow S$ be the blowup of $S$ at the base-point $\mathcal{O}=[1:1:0:0]$ of $|{-K_S}|$ and $\pi: \mathcal{E} \rightarrow \mathbb{P}^1$ be the elliptic fibration given by projection to $[z:w]$. Suppose that there is $P = [x_0:y_0:z_0:w_0] \in S(K)$ with $w_0 \neq 0$, $3z_0 + 2f_2w_0/f_3 \neq 0$ and $f(t) - f(z_0/w_0)$ separable such that $\rho^{-1}(P)$ is non-torsion on its fiber of $\pi$. Then $S$ satisfies the Hilbert property over $K$.
\end{theorem}

\begin{remark}
By \cite[Thm.~3.1]{NIJ}, the existence of a point $P$ as in Theorem~\ref{thm:2} is equivalent to Zariski-density of $S(K)$ since $K$ is finitely generated over $\mathbb{Q}$.
\end{remark}

\begin{remark}
Every del Pezzo surface of degree 1 can be written as a smooth weighted sextic in $\mathbb{P}(2,3,1,1)$, which, over a field $k$ with $\ch(k) \neq 2,3$, is of the form $$y^2=x^3+a_4x+a_6,$$ where $a_i\in k[z,w]$ is homogeneous of degree $i$. Conversely, each such surface is a del Pezzo surface of degree 1. The family in Theorem \ref{thm:2} is thus characterized by $f$.
\end{remark}

We prove Theorem \ref{thm:2} by applying Theorem \ref{thm:1} to the surface $\mathcal E$ endowed with one elliptic fibration $\pi$ and with a certain low-genus family of curves introduced in the aforementioned works of Desjardins and the third author \cite{DW} and generalized by Nijgh \cite{NIJ}, whose construction we sketch in the next paragraph.

The del Pezzo surfaces appearing in the family of Theorem \ref{thm:2} are endowed with a $3$-to-$1$ morphism $\theta:S \to W$ onto a (singular) cubic surface $W$ in $\P^3_K$.  Given a rational point $P \in S(K)$ lying in the inverse image $S^o:=\theta^{-1}(W^{o})$ of the smooth locus $W^o\subset W$, one may consider the intersection curve $D_P := T_{\theta(P)}W \cap W$. For a general $P$, this is a nodal curve in $P$ with geometric genus zero. The arithmetic genus of its inverse image $\theta^{-1}(D_P) \in |{-3K_S}|$ is easily computed to be $4$ via the adjunction formula. Moreover, the curve $C_P:=\theta^{-1}(D_P)$ has three nodal singularities at the (geometric) points $\{P,Q_1,Q_2\}=\theta^{-1}(\theta(P))$, making its geometric genus at most $4-3=1$. The low-genus family is then the associated family $\{\rho^{-1}(C_P)\}_{P\in S^o}$ in $\mathcal E$.

As presented in \cite{NIJ} (following the idea of \cite{DW}), this family may already be directly used to obtain Zariski-density of rational points on $S$. For each $p \in \{P,Q_1,Q_2\}$, the point $[-2]p$ also lies on $C_P$, where $[-2]:S \dashrightarrow S$ denotes the rational map induced by the self-isogeny $[-2]$ of $\mathcal E$. The zero-cycle $[-2](Q_1)+[-2](Q_2) - 2 ([-2]P) \in Z_0(S \otimes_K \bar K)$ then lies on $C_P$ and is $\Gal(\bar K/K)$-invariant, and thus defines a rational point in the Jacobian of (the normalization of) $C_P$. Following \cite{DW} and \cite{NIJ}, this element is non-torsion for $P$ lying outside of a thin set, and thus produces infinitely many rational points on $C_P$. 

The just-sketched argument for Zariski-density proves that the total space $\mathcal C$ of the family $\{C_P\}_{P \in S^o}$ has Zariski-dense rational points, verifying one of the hypotheses of Theorem \ref{thm:1}. The other two hypotheses are verfied in Propositions \ref{prop:generic CR with fibers} and \ref{prop:GIGF}: the first proves that the union $F$ is empty in our case, and the second proves that $\phi:C \to S$ has geometrically integral generic fiber.

Lastly, we show that, as $a_4$, $a_6$ and $f$ vary, the collection of $S$ as in Theorem~\ref{thm:2} with non-thin rational points and Picard rank $1$ is itself non-thin.

\begin{theorem} \label{thm:3}
    For any field $K$ finitely generated over $\mathbb{Q}$, there is a non-thin set of parameters $(a,b,c,d,e,f_0,\ldots,f_3) \in K^9$ such that $S$ as in Theorem~\ref{thm:2} is of Picard rank~$1$ and satisfies the Hilbert property.
\end{theorem}

\subsection{Acknowledgements}
The authors thank Daniel Loughran for useful discussions and both the University of Bath and UniDistance Suisse for their hospitality during two week-long workshops. We also thank David Loeffler for financial support during the workshop at UniDistance Suisse. Lastly we thank Jean-Louis Colliot-Th\'el\`ene for drawing our attention to the work of Manin \cite{M}. This work was supported by a Focused Research Grant from the Heilbronn Institute for Mathematical Research. The second author is supported by the University of Bristol and the Heilbronn Institute for Mathematical Research. Over the course of this project, the third author was supported by UKRI Fellowship MR/T041609/2, ERC Horizon 2020 Consolidator Grant ID 101001051, and MSCA Postdoctoral Fellowship 101148712.\\
\begin{tabular}{c p{13cm}}
\raisebox{-35pt}{ \includegraphics[scale=0.05]{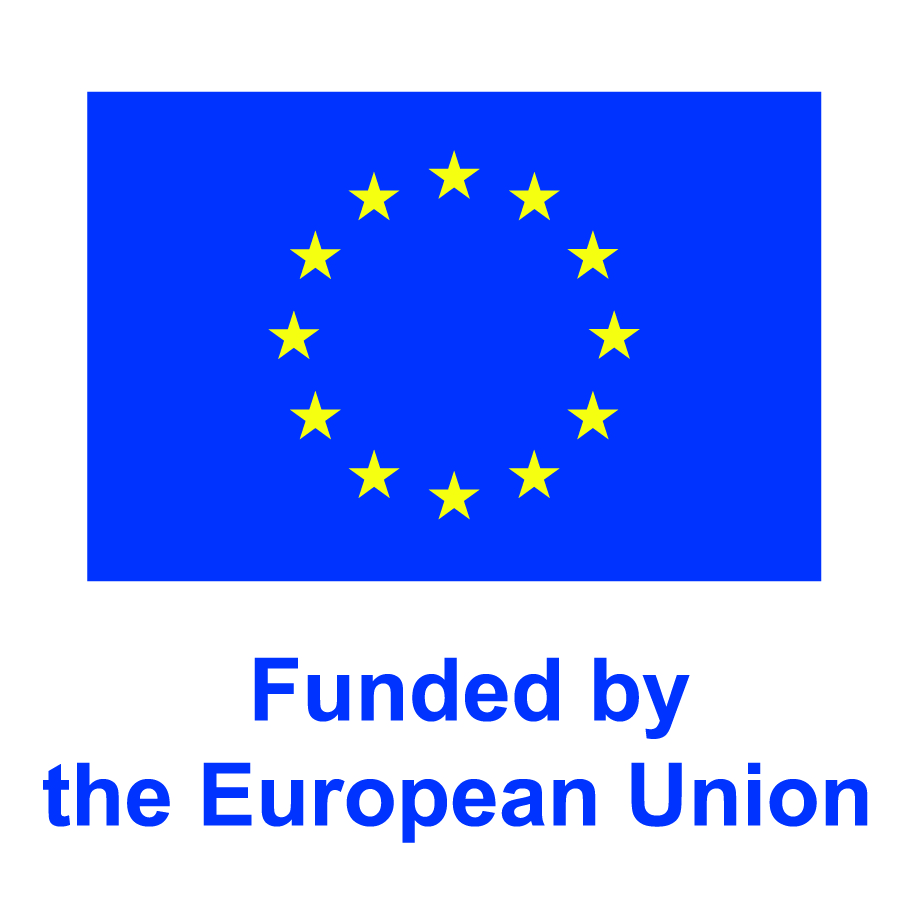}   }&   \tiny{ Funded by the European Union. Views and opinions expressed are however those of the author(s) only
and do not necessarily reflect those of the European Union or the European Research Executive Agency. Neither
the European Union nor the granting authority can be held responsible for them.}
\end{tabular}

\section{Low-genus families and the Hilbert property}\label{sec:ell families an HP}
In this section we prove Theorem~\ref{thm:1} on the Hilbert property for certain elliptic families. We begin with an overview of the Hilbert property.
\subsection{Hilbert property} \label{sec:HP}
\begin{mydef}
Let $X$ be a variety over a field $k$. We say that $A \subset X(k)$ is \emph{thin} if there exist finitely many generically finite dominant morphisms $f_i:Y_i \rightarrow X$, $i=1,\dots,r$, each of degree at least $2$, such that $X(k) \setminus \bigcup_{i=1}f_i(Y_i(k))$ is not Zariski-dense in $X$.
\end{mydef}
Without loss of generality, we can (and do) assume that each $Y_i$ is normal and geometrically integral, and each $f_i$ is finite \cite[Rems.~2.2,~2.3]{STR}.
\begin{mydef}
We say that a variety $X$ over a field $k$ satisfies/has the \emph{Hilbert property} (over $k$) if $X(k)$ is not thin. We say that a field $k$ is \emph{Hilbertian} if there exists a $k$-variety $X$ with the Hilbert property (equivalently, if $\mathbb{P}^1_k$ has the Hilbert property).
\end{mydef}
\begin{example}
Number fields are Hilbertian by Hilbert's Irreducibility Theorem; more generally and crucially for us, any field finitely generated over $\mathbb{Q}$ is Hilbertian \cite[\S13.3]{FJ}. Local fields, finite fields and algebraically closed fields are not Hilbertian. 
\end{example}
We now move on to low-genus families.
\subsection{Low-genus families}

\begin{mydef}
Let $\pi: Y \rightarrow B$ be a surjective flat morphism of varieties over a field $k$. For $g \in \mathbb{Z}_{\geq 0}$, we say that $\pi$ is a \emph{genus-$g$ fibration} if the generic fiber of $\pi$ is a geometrically integral curve of geometric genus $g$. We say that $\pi$ is an \emph{elliptic fibration} if it is a genus-$1$ fibration with smooth generic fiber and a section. If further $Y$ has dimension $2$, then we call $Y$ an \emph{elliptic surface over $B$}.
\end{mydef}
 
We recall that two curves $C$ and $D$ on a smooth surface $S$ are said to intersect {\em transversally} if they intersect transversally at all points, or equivalently, if the intersection scheme $C \cap D$ is reduced. In fact, for Theorem \ref{thm:1} we only need the weaker requirement of one reduced point in the intersection.

\begin{mydef}
Let $\pi\colon Y \rightarrow B$ be a genus-$g$ fibration over a base $B$ with $Y$ regular, and $\psi: Z \rightarrow Y$ be a cover with branch divisor $D$. We say that $\psi$ is \emph{vertically ramified with respect to $\pi$} if $D$ is contained in a union of  fibers of $\pi$ (i.e.\ $D$ is \emph{vertical} with respect to $\pi$), and \emph{horizontally ramified} otherwise.
\end{mydef}

A {\em flat family of proper subschemes} of a variety $Y/k$ with {\em base} $B/k$ is a closed subscheme $X\subset Y \times_k B$, flat over $B$. For such a family, we call $B$ the {\em base} of the family, $Y$ the {\em total space} of the family. Write $p$ for the projection of $Y \times_k B$ to $B$, and $\phi$ for the projection to $Y$. Then the subschemes $X_b=\phi(p^{-1}(b)), b\in B$ the {\em members} of the family. We denote such a family by $\{X_b\}_{b\in B}$. 

\begin{mydef}\label{def:lgfam}
    Let $Y$ be a proper variety over a field $k$. A flat family of proper subschemes $\{X_b\}_{b\in B}$ of $Y$ is a \emph{low-genus family} if its generic element is a geometrically integral curve of geometric genus $0$ or $1$.
\end{mydef}

With the following result, essentially contained in \cite{DEM21}, we reduce the proof of Theorem~\ref{thm:1} to a special case. 

\begin{theorem}\label{thm:redn}
Let $S$ be a surface over a field $K$ finitely generated over $\mathbb{Q}$, and let $\pi_1,\ldots,\pi_n : S \to \mathbb P ^1$ be $n \geq 1$ distinct genus-$1$ fibrations. Denote by $\mathcal{C}$ the set of finite covers $\psi: Y \rightarrow S$ of degree at least $2$ which are vertically ramified with respect to each $\pi_i$.
If, for every finite subset $\{\psi_j:Y_j \to S, i=j,\ldots,r\} \subset \mathcal C$, the set
\[
S(K) \setminus \cup_{j=1}^r \psi_j(Y_j(K))
\]
is Zariski-dense in $S$, then $S$ satisfies the Hilbert property.
\end{theorem} 

\begin{proof}
The proof of \cite[Thm.~1.1]{DEM21} shows that, for $K$ a number field, given covers $\psi_j: Z_j \rightarrow S$, $j = 1,\dots,s$ of degrees $\geq 2$ such that $S(K) \setminus \bigcup_{j=1}^s\psi_j(Z_j(K))$ is not dense, there is a subset $\{j_1,\ldots,j_t\}\subset\{1,\ldots,s\}$ such that $\psi_{j_i}: Z_{j_i} \rightarrow S$ is contained in $\mathcal{C}$ for all $j_i\in\{j_1,\ldots,j_t\}$ and $S(K) \setminus \bigcup_{i=1}^t\psi_{j_i}(Z_{j_i}(K))$ is not dense (the simply connected hypothesis in \emph{loc.\ cit.}\ ensures that $\mathcal{C} = \emptyset$ in that setting). Since this non-density would contradict our hypothesis, we deduce the result for number fields. To deduce the same for $K$ finitely generated over $\mathbb{Q}$, note that the Mordell--Weil theorem, Faltings' theorem and Merel's theorem, which are the key ingredients in the proof of \emph{loc.\ cit.}\ in the setting of number fields, all admit generalizations to such $K$: the generalization of Mordell--Weil follows from the Lang--N\'eron theorem (see \cite[Cor.~7.2]{CON}), the generalization of Faltings' theorem follows from the work of Faltings--W\"ustholz \cite[Thm.~3,~p.~205]{FW}, and the generalization of Merel's theorem is established by Colliot-Th\'el\`ene in \cite[Appendix]{DW}.
\end{proof}

We are now ready to prove Theorem \ref{thm:1}.

\begin{proof}[Proof of Theorem \ref{thm:1}]
Define $\mathcal{C}$ to be the set in Theorem~\ref{thm:redn} if $n \geq 1$ or the set of all finite covers of $S$ of degree $\geq 2$ if $n=0$, and let $\{\psi_i:V_i \rightarrow S, i=1,\dots,s\}$ be a finite subset of $\mathcal{C}$. We will show that $S(K) \setminus \bigcup_{i=1}^r\psi_i(V_i(K))$ is Zariski-dense in $S$, so that we can deduce the result from Theorem~\ref{thm:redn}. For each $i \in \{1,\dots,s\}$, denote by $F_i$ the fiber product $X \times_{\phi,\psi_i} V_i$. Let $X_{\eta}$ (resp.\ $F_{i,\eta}$) be the generic fiber of the low-genus fibration $\pi:X \rightarrow B$ (resp.\ $F_i \to X \rightarrow B$). 

Our assumption that the generic fiber of $\phi$ is geometrically integral implies that for all $i \in \{1,\dots,s\}$, the curve $F_{i,\eta}$ is integral over $K(B)$. In fact, consider the following commutative diagram.
\[
\begin{tikzcd}
{F_{i,\eta}} \arrow[d] \arrow[r] & V_i \arrow[d, "\psi_i"] \\
X_{\eta} \arrow[r, "\phi"]  
& S.                
\end{tikzcd}
\]
Since $S$ is a smooth surface and $V_i$ is normal, the morphism $\psi_i$ is flat by miracle flatness \cite[Tag 00R4]{SP}. It follows that $F_{i,\eta} \to X_{\eta}$ is finite flat as well (both finiteness and flatness being preserved under base change), with generic fiber $\Spec K(X) \times_{\Spec K(S)} \Spec K(V_i) \cong \Spec (K(X) \otimes_{K(S)} K(V_i))$. Since $\phi$ has geometrically integral generic fiber, the extension $K(X)/K(S)$ is regular and thus linearly disjoint from the algebraic extension $K(V_i)/K(S)$. Hence $K(X) \otimes_{K(S)} K(V_i)$ is a field. In particular, the flat morphism $F_{i, \eta} \to X_{\eta}$ has integral generic fiber, and thus $F_{i, \eta}$ is also integral \cite[Prop.~4.3.8]{LIU}. 

For each $i\in\{1,\ldots,s\}$, denote by $D_i \subset S$ the branch locus of $\psi_i$, which is a divisor by Nagata--Zariski purity \cite[Tag~0BMB]{SP}. We claim that a general member of the low-genus family intersects each $D_i$ transversally: otherwise, there is at least one $i_0$ such that $D_{i_0}$ belongs to $F$, but then the restriction of $\psi_{i_0}$ over $S \setminus F$ is an unramified cover of degree $\geq 2$, contradicting that $S \setminus F$ is algebraically simply connected. Then there is at least one smooth point on $F_{i,\eta}$ which is ramified over $X_{\eta}$ (cf.~\cite[p.~595]{CZ}), hence $F_{i,\eta}$ is an integral $K(B)$-curve ramified over  $X_{\eta}$. We may spread out over some open $U \subset B$ to a morphism $\xi_i\colon F_{i,U} \to X_U$ such that, for all $P \in U$, the curve $F_{i,P}$ is a ramified cover of $X_P$.

Since $K$ is finitely generated over $\mathbb{Q}$ and $X(K)$ is Zariski-dense,  there is a non-empty Zariski open $V \subset X_U$ such that, for each $v \in V(K)$, the fiber $\pi^{-1}(\pi(v))$ has infinitely many $K$-rational points; see \cite[Appendix]{DW} for the genus-$1$ case (the genus-$0$ case is elementary, as any smooth curve of genus $0$ with one rational point has infinitely many). By assumption, $V(K)$ is Zariski-dense in $V$, and we deduce Zariski-density in $U$ of
\[
\mathcal A := \{u \in U(K): \#\pi^{-1}(u)(K)=\infty\}.
\]

When $g = 1$, it follows from Lang--N\'eron that, for each $u \in \mathcal{A}$, only finitely many points in $\pi^{-1}(u)(K)$ can lift to $F_{i,U}$ for some $i$, hence $\pi^{-1}(u)(K) \setminus \cup_{i=1}^s \xi_i(F_{i,U}(K))$ is dense in $\pi^{-1}(u)$; the same conclusion holds for $g = 0$ since $\pi^{-1}(u)$ has the Hilbert property, hence
\[
\mathcal A' :=\bigcup_{u \in \mathcal A} \pi^{-1}(u)(K) \setminus \cup_{i=1}^s \xi_i(F_{i,U}(K))
\]
is Zariski-dense in $X_U$ in either case. Thus $\phi(\mathcal A')$ is Zariski-dense in $S$. Since $\phi(\mathcal A')$ is contained in $S(K) \setminus \cup_{i=1}^s \psi_i(V_i(K))$, this concludes the proof.
\end{proof}

In the following remark we argue on the necessity of the hypothesis of geometric integrality of the generic fiber of $\phi$.

\begin{remark}\label{rem:hypothesis is necessary}
We make some observations on the applicability of Theorem~\ref{thm:1}.
\begin{enumerate}
\item A special case of Theorem \ref{thm:1} is that with no genus-$1$ fibrations ($n=0$). In this case, our proof shows that in fact $\phi(X(K))$ is non-thin in $S(K)$. 
    
\item We give an example of a surface $S$ and a family of genus-$0$ curves $\phi:X \to S$ with fibration $X\to B$ whose general member intersects any given divisor $D$ on $S$ transversally, and yet $\phi(X(K))$ is thin in $S(K)$. In this example, the generic fiber of $\phi$ is not geometrically integral, and so Theorem \ref{thm:1} (and point (1) of this remark) is not applicable.
    
    Take a general cubic pencil on $\P^2_K$, let $T$ be its base locus ($T$ consists of $9$ geometric points), and enlarge $K$ so that at least one of the points of $T$ is rational, call it $t$. Let $S$ be the blow-up of $\P^2_K$ in $T$, and let $E$ be the exceptional divisor above $t$. The proper transform of the cubic pencil makes $S$ into an elliptic surface with zero-section $E$. Let now $B$ be a non-empty Zariski-open of the dual projective space $(\P^2_K)^{\vee}$ containing none of the lines passing through the (geometric) points of $T$. Let $\{l_b\}_{b \in B}$ be the family parametrizing the proper transforms in $S$ of the corresponding lines, and let $\{X_b := [2]l_b\}_{b \in B}$ be the image of this family under the rational map $[2]:S \dashrightarrow S$ (we restrict $B$ so that $l_b$ does not pass through any of the finitely many points not contained in the domain of $[2]$). Denote by $X$ the total space of the family $\{X_b\}_{b \in B}$. It is clear that $\phi(X(K)) \subset [2](S(K))$ is thin. As we verify in the next paragraph, the general member $X_b$ intersects any given divisor $D$ on $S$ transversally, and thus this family provides an example as desired.

    Assume without loss of generality that $D$ is integral. The curve $X_b=[2]l_b$ intersects $D$ transversally if and only if $l_b$ intersects $[2]^*D$ transversally and $l_b$ contains no pair of distinct points $q_1,q_2 \in S(\overline{K})$ with $[2](q_1)=[2](q_2)\in D$. The general line $l_b \in (\P^2_K)^{\vee}$ satisfies both these conditions: indeed the map $[2]:S \dashrightarrow S$ is unramified on its domain of definition, so the divisor $[2]^*D$ is reduced, and thus the general line intersects it transversally, while the second condition describes a one-dimensional set of lines.
\end{enumerate}    
\end{remark}

\begin{remark}
    Theorem \ref{thm:1} still holds if one replaces $F$ with the set of divisors that are vertical with respect to each $\pi_i$ and do not intersect a general member of the family $\{X_b\}_{b\in B}$ transversally in at least one point. Indeed, observe that in the proof of Theorem \ref{thm:1} we did not need the general member of the family $\{X_b\}_{b \in B}$ to intersect the divisors in $F$ transversally {\em at all intersections points}, but  it was rather enough that it intersects these divisors transversally {\em in at least one point}. 
\end{remark}

\section{Low-genus families on degree-one del Pezzo surfaces}\label{sec:ell families and dPs}

We now introduce the family of degree-one del Pezzo surfaces studied by Nijgh \cite{NIJ}. Throughout this section, let $K$ be a field that is finitely generated over $\mathbb{Q}$.
\begin{mydef}\label{DefS}
Let $f(t)=f_3t^3 + f_2 t^2 + f_1 t + f_0 \in K[t]$ be a cubic polynomial and define
\[
a_4(t) := at + b, \quad a_6(t):=ct^2 + dt + e \in K[t].
\]
Let $S \subset \mathbb{P}(2,3,1,1)$ be the surface given by
\[
S: y^2  = x^3 + a_4(f(z/w))x w^4 + a_6(f(z/w))w^6.
\]
The surface $S$ is smooth if and only if the sextic branch curve $x^3 + a_4(f(z/w))x w^4 + a_6(f(z/w))w^6=0$ in $\P(2,1,1)$ is. We assume that $S$ is smooth from now on, hence a del Pezzo surface of degree one. Recall that we have $\rho: \mathcal{E} \rightarrow S$ the blowup of $S$ at the base-point $\mathcal{O}=[1:1:0:0]$ of $|{-K_S}|$ and $\pi: \mathcal{E} \rightarrow \mathbb{P}^1$ the associated elliptic fibration.
\end{mydef}

\begin{note}
The above family corresponds to the one studied by Nijgh in \cite{NIJ} via linear change of variables. Further, it strictly contains the family
\[
y^2 = x^3 + cz^6 + dz^3 w^3 + ew^6,
\]
studied in \cite{DW}, recovered by the choices $f(t) = t^3$, $a_4 = 0$ and $a_6 = ct^2 + dt + e$.
\end{note}

Consider the regular map $\theta:S \to \P^3, [x:y:z:w] \mapsto [xw:y:w^3f(z/w):w^3]$. The image $W:=\theta(S)$ is a singular cubic surface in $\mathbb{P}^3$ with equation
\[
\begin{aligned}
W:  X_1^2 X_3= X_0^3 + a_4(X_2/X_3)X_0 X_3^2 + a_6(X_2/X_3) X_3^3.
\end{aligned}
\]

\begin{lemma}\label{Lem:SING}
The following hold.
\begin{enumerate}
    \item The singular locus of $W$ is $\{[0:\pm \sqrt{c}:1:0]\}$. 
    \item The types of singularities on $W$ are as follows:
    \begin{enumerate}
    \item If $c \neq 0$, we obtain two $A_2$ singularities.
    \item If $c = 0$ and $a \neq 0$, we obtain one $A_5$ singularity.
    \item If $c = a = 0$, we obtain one $E_6$ singularity.
    \end{enumerate}
    \item The cubic surface $W$ is not a cone.
    \end{enumerate}
\end{lemma}
\begin{proof}

\begin{enumerate}
\item Let $W^{\text{aff}} \subset \A^3$ and $S^{\text{aff}} \subset \A^3$ be the restrictions of $W$ and $S$ to the affine charts $X_3 \neq 0$ and $w \neq 0$, respectively given by $x_1^2 = x_0^3 + a_4(x_2)x_0 + a_6(x_2)$ and $y^2=x^3+a_4(f(z))x+a_6(f(z))$. Note that $S^{\text{aff}} = \phi^{-1}(W^{\text{aff}})$ for the flat map
\[
\phi:\A^3 \to \A^3, \quad (x,y,z) \mapsto (x,y,f(z)),
\]
and the induced map $S^{\text{aff}}\to W^{\text{aff}}$ is also flat by preservation of flatness under base-change. 
Now, any variety $V$ with a flat surjective morphism $V' \to V$ from a smooth variety $V'$ is smooth (see \cite[Thm.~23.7]{Matsumura}), hence $W^{\text{aff}}$ is smooth. Then all singularities of $W$ lie on $X_3 = 0$.

When $X_3=0$, the equation for $W$ also gives $X_0=0$. The point $[0:1:0:0]$ is smooth, as $(\partial F_W/\partial X_3)(0,1,0,0)= -1$, with $F_W := X_0^3 + a_4(X_2/X_3)X_0 X_3^2 + a_6(X_2/X_3) X_3^3 - X_1^2 X_3$. The remaining points on the line $X_0 = X_3 = 0$ are covered by the affine chart $X_2 \neq 0$, where the equation becomes $x_1^2 x_3= x_0^3 + a_4'(x_3)x_0 x_3 + a_6'(x_3) x_3,$ with $a_4', a_6' \in k[t]$ defined by $a_4'(t)=ta_4(1/t)$ and $a_6'(t)=t^2a_6(1/t)$. We rewrite the equation as $x_1^2x_3-cx_3=O(x_0,x_3)^2$ and deduce that the only singular points with $x_0=x_3=0$ are $(0,\pm \sqrt{c},0)$.

\item To determine the singularity type/types, we follow \cite[\S2]{BW}.

\begin{enumerate}
\item First assume that $c \neq 0$.
    \begin{enumerate}
    \item If further $a \neq 0$, then
    \[
    \begin{aligned}
    & F_W\left(\frac{2\sqrt{c}}{a}\left(X_0-\frac{d}{2\sqrt{c}} - X_2\right),\frac{1}{2}(-X_2+X_3),\frac{1}{2\sqrt{c}}(X_2 + X_3), X_1\right) \\
    & = X_0 X_1 X_3 + G_W(X_0,X_1,X_2),
    \end{aligned}
    \]
    \[
    \begin{aligned}
    G_W(X_0,X_1,X_2) & := \frac{8c\sqrt{c}}{a^3}X_0^3 - \frac{12cd}{a^3}X_0^2 X_1 + \frac{2a^2b\sqrt{c} + 6\sqrt{c}d^2}{a^3}X_0 X_1^2 \\
    & + \frac{a^3 e - a^2 b d - d^3}{a^3}X_1^3 - \frac{24c\sqrt{c}}{a^3} X_0^2 X_2 \\
    & + \frac{a^3 + 24cd}{a^3}X_0 X_1 X_2 + \frac{-2a^2b\sqrt{c} - 6\sqrt{c}d^2}{a^3}X_1^2 X_2 \\
    & + \frac{24 c\sqrt{c}
    }{a^3} X_0 X_2^2 + \frac{-a^3 - 12 c d}{a^3}X_1 X_2^2 - \frac{8 c\sqrt{c}}{a^3} X_2^3.
    \end{aligned}
    \]
    Since $G_W(0,0,1) \neq 0$, one of the singularities is $A_2$ by \cite[Lem.~3]{BW}, hence so is the other singularity by symmetry.

    \item If instead $a = 0$, then
    \[
    \begin{aligned}
    & F_W\left(X_2,\frac{1}{2}\left(-X_0 + \frac{d}{2\sqrt{c}}X_1 + X_3\right),\frac{1}{2\sqrt{c}}\left(X_0 - \frac{d}{2\sqrt{c}}X_1 + X_3\right),X_1\right) \\
    & = X_0 X_1 X_3 + G_W(X_0,X_1,X_2),
    \end{aligned}
    \]
    \[
    \begin{aligned}
    G_W(X_0,X_1,X_2) & := \frac{d}{2\sqrt{c}}X_0 X_1^2 + \frac{4ce - d^2}{4c}X_1^3 + bX_1^2 X_2 + X_2^3.
    \end{aligned}
    \]
    Since $G_W(0,0,1) \neq 0$, we reach the same conclusion as in (i).
    \end{enumerate}

\item Next suppose that $c = 0$ and $a \neq 0$. Then
\[
F\left(\frac{1}{a}(X_0-dX_1),X_2,X_3,X_1\right) = X_0X_1X_3 + G_W(X_0,X_1,X_2,X_3),
\]
\[
\begin{aligned}
G_W(X_0,X_1,X_2) & := \frac{1}{a^3}X_0^3 - \frac{3d}{a^3}X_0^2 X_1 + \frac{a^2b + 3d^2}{a^3}X_0X_1^2 \\
& + \frac{a^3e - a^2bd - d^3}{a^3}X_1^3 - X_1 X_2^2.
\end{aligned}
\]
Since $g_1(X_1) := G_W(0,X_1,1)$ has a zero of order $1$ at $X_1 = 0$ and $g_0(X_0):=G_W(X_0,0,1)$ has a zero of order $3$ at $X_0 = 0$, the singularity is $A_5$ \cite[Lem.~3]{BW}.
    
\item Now assume that $c = a = 0$; by smoothness of $S$, we find 
$d \neq 0$. Then 
\[
F\left(X_2,X_1,X_3,\frac{1}{\sqrt{d}}X_0\right) = X_3 X_0^2 + G_W(X_0,X_1,X_2),
\]
\[
G_W(X_0,X_1,X_2) := X_2^3 + \frac{1}{d\sqrt{d}} X_0^2 \left(e X_0 + b\sqrt{d}X_2\right) - \frac{1}{\sqrt{d}} X_0 X_1^2.
\]
Since $G_W(0,X_1,X_2) = X_2^3$, the singularity is $E_6$ by \cite[Lem.~4]{BW}.
\end{enumerate}
\item If $W$ were a cone, then since all of its singularities are isolated by the above, it would have to be a cone over a smooth plane cubic $C$. Denote by $\phi:W \dashrightarrow C$ the natural projection. The composition $\phi \circ \theta:S \dashrightarrow C$ is a dominant rational map from $S$ to the curve $C$ of genus $1$. However, since $S$ is a geometrically rational surface, such a map cannot exist and this is a contradiction. \qedhere
\end{enumerate}
\end{proof}

We give an explicit description of the ramification locus of $\theta$:
\begin{lemma}\label{Lem:BRANCH}
    The ramification locus of $\theta$ is the union of the two (possibly coinciding) points $[0:\pm f_3 \sqrt c:1:0]$ and the two (possibly coinciding) divisors $\{z=\rho_1 w\},\{z=\rho_2 w\}$, where $\rho_1$ and $\rho_2$ denote the roots of $f'(t)$.
\end{lemma}

Here $f'(t)=3f_3t^2+2f_2t+f_1$ indicates the derivative of $f$. The points $[0:\pm f_3 \sqrt c:1:0]$ are the points of $S$ lying above the singular points $[0:\pm \sqrt{c}:1:0]$ of $W$. 

\begin{proof}
    The restriction of $\theta$ to $S^{\text{aff}}$ is the restriction of the map $\phi:\A^3 \to \A^3:(x,y,z) \mapsto (x,y,f(z))$. The map $\phi$ ramifies precisely when $f'(z)=0$, i.e.\ on the two planes defined by $z=\rho_1,z=\rho_2$. Since the ramification locus of a morphism is invariant under arbitrary base-change (the ramification locus is the support of the sheaf of relative differentials, so this follows from the fact that the latter is invariant under arbitrary base-change \cite[Tag 01V0]{SP}), and the morphism $S^{\text{aff}}\to W^{\text{aff}}$ is the base-change of $\phi$ along the inclusion $W^{\text{aff}} \subset \A^3$, the ramification locus of $\theta|_{S^{\text{aff}}}=\phi|_{S^{\text{aff}}}:S^{\text{aff}}\to W^{\text{aff}}$ are then the two plane sections defined by $z=\rho_1,z=\rho_2$ of $S^{\text{aff}}$. Thus the ramification locus of $\theta$ contains their projective closures $\{z=\rho_1 w\},\{z=\rho_2 w\}$.

    Let now $D:\{w=0\}$ be the divisor at infinity of $S$. It remains to determine on which points of $D$ the map $\theta$ ramifies. Looking at the expression for $\theta$, we see that $D$ is the scheme-theoretic inverse image $\theta^{-1}(l)$ of the line $l:\{X_0=X_3=0\}$. Hence the morphism
    \[
    \alpha=\theta|_D:D\to l, \ [x:y:z:0]\mapsto [0:y:f_3z^3:0]
    \]
    is the base-change of $\theta$ along the inclusion $l\subset W$. Again applying preservation of the ramification locus under base-change, the points of $D$ where $\theta$ ramifies coincide with those where $\alpha$ ramifies. The point $[1:1:0:0]$ is contained in both ramification divisors listed above, and so it is a ramified point. The complement $D\setminus \{[1:1:0:0]\}$ is fully contained in the chart $z\neq 0$, where $D$ restricts to the affine Weierstrass curve $y^2=x^3+cf_3, w=0$, and $\alpha$ restricts to the projection $(x,y,0)\mapsto (0,y,0)$, which ramifies precisely at the points $(0,\pm f_3 \sqrt c,0)$, giving the points $[0:\pm f_3 \sqrt c:1:0]$ in the weighted projective space.
\end{proof}

Let $W^{o} := W \setminus \{[0:\pm \sqrt{c}:1:0]\}$ be the smooth locus of $W$, and set $S^o := \theta^{-1}(W^o)$.

\begin{mydef}
We define the following curves:
\begin{enumerate}
\item For $P \in S^o$, we define $C_P:=\theta^{-1}(T_{\theta(P)}W \cap W)$ for $T_{\theta(P)} W$ the tangent plane.
\item For $P\in \mathcal{E}$, we set $E_P := \pi^{-1}(\pi(P)) \subset \mathcal{E}$.
\item For $P \in S\setminus \{\mathcal{O}\}$, we define $D_P = \rho\left(E_{\rho^{-1}(P)}\right) \subset S$ to be the unique curve in $|{-K_S}|$ through $P$ and $B_P:=\theta(D_P) \subset W$.
\end{enumerate}
\end{mydef}

\begin{remark}\label{rem:properties of C_P}
For $P \in S^o$, we have the following properties.
\begin{enumerate}
\item $C_P \sim -3K_S$ \cite[Lem.~3.9]{NIJ}.
\item When irreducible, $C_P$ has geometric genus at most $1$: for general $P$, the curve $C_P$ has (at least) three double points at $\theta^{-1}(\theta(P))$, while $C_P \in |{-3K_S}|$ has arithmetic genus 4 by adjunction. For non-generic $P$, the geometric genus is then still bounded by $1$ by semi-continuity of the genus in flat families.
\item Viewing $C_P$ as a divisor on $S$, its restriction $C_P|_{D_P}$ to $D_P$ is $(2)P+Q$, where $Q = [-2]P$ on $D_P$ as an elliptic curve with origin $\mathcal{O}$. Indeed, $P$ appears with multiplicity at least $2$ as it is a singularity of $C_P$, while the expression for $Q$ follows by noting the rational equivalences $C_P|_{D_P}\sim (-3K_S)|_{D_P} \sim 3\mathcal O$.
\end{enumerate}
\end{remark}

In view of our wish to apply Theorem \ref{thm:1} to the fibration $\pi$ and the family $\{C_P\}_{P \in S^o}$, we show that the general member of the family intersects the curves of the pencil $|{-K_S}|$ transversally.

\begin{proposition}\label{prop:generic CR with fibers}
    For each $D \in |{-K_S}|$, the set
    $$
    \{R\in S^o: |C_R\cap D|<3\}
    $$ 
    is a proper closed subset of $S^o$.
\end{proposition}

\begin{proof}
    The set is clearly closed, so it suffices to find one point $R \in S$ not lying inside it.

     We treat first the case where $D$ is not the divisor $\{w=0\}\in |{-K_S}|$. The curve $D$ is irreducible, thus so is $B=\theta(D)$. This image $B$ coincides with the plane section of $W \subset \mathbb{P}^3$ given by $\Pi: w_0^3 X_2 - w_0^3f(z_0/w_0)X_3 = 0$. 

    Let $L \subset \Pi$ be a general line. This intersects $B$ transversely in 3 points. Let $\sigma: W' \rightarrow W$ be the composition of the minimal desingularization of $W$ and the blowup in the inverse image of $L \cap W$. Let $g: W' \rightarrow \mathbb{P}^1$ be the genus-$1$ fibration corresponding to the pencil of plane sections cut out in $W$ by planes containing $L$.

    Any relatively minimal model of $W'$ is a rational elliptic surface, which has (étale/\linebreak topological) Euler number $12$ \cite[Ch.~5]{SCSH}, making the Euler number $e$ of $W'$ at least $12$. The Euler number of a regular genus-$1$ fibration is the sum $\sum_v e(F_v)$ of the Euler numbers of its singular fibers $F_v$ ({\em loc.\ cit.}). 
    Consider now the planes $\Pi_{\pm}$ passing through $L$ and $[0:\pm \sqrt c:1:0]$, with the convention that  $\Pi_+=\Pi_-$ in the case $c=0$. Let $F,F_{\pm}$ be the fibers corresponding to the plane sections $\Pi\cap W$ and $\Pi_{\pm} \cap W$, respectively. 

    The diagram below summarizes part of what we defined so far. 
    \[
        \begin{tikzcd}
    D \arrow[d, phantom, sloped, "\subset"]\arrow[r, "\theta"]         & B=\Pi\cap W             \arrow[d, phantom, sloped, "\subset"]&\arrow[l, "\sigma" above, "\sim" below] F\arrow[d, phantom, sloped, "\subset"]\\
    S\arrow[r, "\theta"]& W &W' \arrow[l, "\sigma" above] \arrow[d, "g"]\\
        &  & \mathbb{P}^1
    \end{tikzcd}
    \] 
    We know that $B$ is irreducible. We claim that each $\Pi_{\pm} \cap W$ is also irreducible. In fact, since $L$ is a general line on $\Pi$, it is skew with respect to any finite set of lines in $\mathbb{P}^3$ not lying on $\Pi$. There are only finitely many lines on $W$ since it is not a cone (Lemma~\ref{Lem:SING}), and none of these lies on $\Pi$ since $\Pi \cap W=B$ is irreducible, hence $L$ is skew with respect to all lines of $W$. Then $\Pi_{\pm} \cap W$ are plane cubics with no lines, proving the claim. 
    
    From irreducibility above, we infer that $F \cong \Pi\cap W$ has $1$ component, while each $F_{\pm}$ has at most $3$ components when $F_+ \neq F_-$ (equiv.\ $c \neq 0$) and $6$ or $7$ components when $F_+ = F_-$ (equiv.\ $c = 0$): the proper transform of $\Pi_{\pm}\cap W$ and the exceptional components coming from the desingularization (the numbers come from Lemma~\ref{Lem:SING}: $A_k$ and $E_k$-singularities give $k$ exceptional components). In each case we have $\sum_{F_v \in \{F_+,F_-\}}e(F_v) \leq 8$, thus $\sum_{F_v \in \{F,F_+,F_-\}}e(F_v) \leq 10$, and there is at least one other singular fiber $F_Q \neq F,F_{\pm}$. 
    
    Now $\sigma(F_Q)$ is a plane section of $W$ with some singular point $R' \neq [0: \pm \sqrt c:1:0]$, hence $\sigma(F_Q) = T_{R'} W \cap W$. Then, for $R \in \theta^{-1}(R')$, the curve $C_R = \theta^{-1}(\sigma(F)) \subset S$ intersects $D$ in at least 3 distinct points. Since $C_R \cdot D = 3$, we are done. 
    
    Finally, let us consider the case $D = \{w=0\}$. In this case, looking at the expression for $\theta$ we see that $\theta(D)$ is the line $l = \{X_0=X_3=0\}$ and that $D=\theta^{-1}(l)$ (as already noticed in the proof of Lemma \ref{Lem:BRANCH}). Take now any $R$ such that $\theta(R)$ does not lie on any of the finitely many lines on $W$ passing through its singular points $\{[0:\pm \sqrt{c}:1:0]\}$. With this choice of $R$, the plane $T_{\theta(R)}W$ does not contain any of the points $\{[0:\pm \sqrt{c}:1:0]\}$. Indeed, if it did, then the line passing through $\theta(R)$ and this point would intersect $W$ with multiplicity $\geq 4$, with a contribution of at least $2$ both from $\theta(R)$ and the singular point.
    
    We conclude by observing that the inverse image $C_R = \theta^{-1}(T_{\theta(R)}W\cap W) = \theta^{-1}(T_{\theta(R)}W)$ intersects $D=\theta^{-1}(l)$ in the subscheme $\theta^{-1}(l\cap T_{\theta(R)}W)$, and the latter has at least $3$ distinct geometric points by Lemma \ref{Lem:BRANCH}.    \qedhere
\end{proof}

Although not needed in this paper, we include the following for completeness:

\begin{lemma}
    For general $P$, the curve $C_P$ has geometric genus $1$.
\end{lemma}

\begin{proof}
    Let $P\in S(\bar K)$ be general. 
    We first observe that, as argued at the end of Proposition \ref{prop:generic CR with fibers}, the tangent plane $T_{\theta(P)}W$ does not contain any of the singular points of $W$. 
    Thus, apart from the three simple nodal singularities in $\theta^{-1}(P)$, the only other singularities the inverse image $\theta^{-1}(T_{\theta(P)}W\cap W)$ can acquire are at the points of non-transverse intersection of the curve $T_{\theta(P)}W\cap W$ with the branch locus of $\theta$ (at transverse intersection points the inverse image is smooth, e.g.\ this can be seen using the aforementioned argument presented in \cite[p.~595]{CZ}).

    Lemma \ref{Lem:BRANCH} tells us that the branch locus $B$ of $\theta$ is the union $\theta(\{z=\rho_1w\})\cup \theta(\{z=\rho_2w\})\cup W^{\text{sing}}$, where $\rho_1$ and $\rho_2$ denote the roots of $f'(t)=0$. The divisors $\{z=\rho_iw\}_{i=1,2}$ are part of the pencil $|{-K_S}|$ and neither of them is the curve $\{w=0\}\in |{-K_S}|$. The argument of the first part of the proof of Proposition \ref{prop:generic CR with fibers} gives that the curve $T_{\theta(P)}W\cap W$ is transverse to both of them, concluding the proof.
\end{proof}

We introduce the total space $\mathcal{C}$ for the family $\{C_P\}_{P \in S^o}$.

\begin{mydef}
Let $\mathcal{C} \subset S^o \times S$ be the incidence variety whose $\overline{K}$-points are
\[
\mathcal{C}(\overline{K}) = \{(P,Q) \in S^o(\overline{K}) \times S(\overline{K}): Q \in C_P\}.
\]
Denote by $p: \mathcal{C} \rightarrow S^o$ and $\phi: \mathcal{C} \rightarrow S$ the two projections. Define the normalization morphism $\nu:\mathcal{C}^{\text{n}} \to \mathcal{C}$, and let $p^{\text{n}}:= p \circ \nu, \phi^{\text{n}}:= \phi \circ \nu$ be the two compositions. 
\end{mydef}

\begin{proposition}
    The projection $p: \mathcal{C} \rightarrow S^o$ is flat and $\mathcal C$ is geometrically integral.
\end{proposition}
\begin{proof}
    The fibers of $p$ are the members of the family $\{C_P\}_{P \in S^o}$. These curves lie in $|{-3K_S}|$, and in particular they all have the same Hilbert polynomial (with respect to any fixed ample line bundle on $S$), making $p$ flat \cite[Thm.~III.9.9]{HAR}. For general $P$, the curve $C_P$ is geometrically integral \cite[Prop.~3.16]{NIJ}. The morphism $p:\mathcal{C} \to S^o$ is now flat over a geometrically integral base with geometrically integral generic fiber, hence $\mathcal{C}$ is geometrically integral \cite[Prop.~4.3.8]{LIU}.
\end{proof}

Recall the blow-up map $\rho\colon\mathcal{E}\rightarrow S$. As already sketched in the introduction, the works \cite{DW} and \cite{NIJ} give:
\begin{proposition}\label{prop:points on T are dense}
The set of rational points $\mathcal{C}(K)$ on $\mathcal{C}$ is dense.
\end{proposition}
\begin{proof}
The proof of \cite[Thm.~3.1]{NIJ} shows that, for infinitely many $R \in \mathbb{P}^1(K)$, the surface $\mathcal{C}_R:=\mathcal{C} \times_{p,S}D_R$ (i.e.\ the restriction of $\mathcal{C}$ to $P \in D_R$) has dense $K$-rational points. In particular, $\{R \in \mathbb{P}^1(K): (\pi\circ\rho^{-1}\circ p)^{-1}(R)(K) \text{ is dense in } (\pi\circ\rho^{-1}\circ p)^{-1}(R)\}$ is dense in $\mathbb{P}^1(K)$, thus $\mathcal{C}(K)$ is dense in $\mathcal{C}$.
\end{proof}

\begin{proposition}\label{prop:GIGF}
The morphism $\phi: \mathcal{C} \rightarrow S$ has geometrically integral generic fiber.
\end{proposition}

\begin{proof}
Assume that $\phi$ does not have a geometrically integral generic fiber; then the same holds for $\phi^{\text{n}}$, and we get a relative normalization decomposition:
\begin{equation}\label{Eq:Stein}
    \phi^{\text{n}}:\mathcal{C}^{\text{n}} \xrightarrow{g} S' \xrightarrow{h} S,
\end{equation}
 where $S'$ is integral, $g$ has geometrically integral fibers, and $h$ is finite of degree at least $2$. We are going to derive a contradiction by showing that the function field extension $K(S')/K(S)$ is a subextension of two linearly disjoint extensions of $K(S)$.

\vskip1mm

{\bf First extension.} The map $p$ admits a rational section $S \dashrightarrow \mathcal{C}, \ P \mapsto (P,Q)$ with $Q \in D_P$ such that $Q=[-2]P$ on the elliptic curve $D_P$ with origin $\mathcal{O}$ (see Remark~\ref{rem:properties of C_P}). For general $P$, this point $Q$ has multiplicity $1$ in the intersection $D_P \cap C_P$ and is thus smooth in $C_P$, hence $(P,Q)$ lies in the smooth locus of $\mathcal{C}$. In particular, this section lifts to a map $s:S \dashrightarrow \mathcal{C}^{\text{n}}$. The composition $\phi^{\text{n}} \circ s: S \dashrightarrow S$ is the map $[-2]:S \dashrightarrow S$, and thus $K(S')$ is contained in the degree-$4$ extension $K(S)_2/K(S)$ (with $K(S)_2 \cong K(S)$ as abstract fields) given by multiplication by $[-2]$.

\vskip1mm

{\bf Second extension.}
For general $P$, the restriction of the normalization morphism $\nu:\mathcal{C}^{\text{n}} \to \mathcal C$ above $C_P \cong \{P\} \times C_P \subset \mathcal C$ is the normalization morphism $\nu_P:C_P^{\text{n}} \to C_P$ of $C_P$. Since $C_P$ has a simple node at $P$, the inverse image $\nu_P^{-1}(P)$ comprises two (reduced) geometric points. It follows that the inverse image $V$ in $\mathcal C^{\text{n}}$ of a sufficiently small non-empty Zariski-open $U$ of the diagonal $\Delta:= \{(P,P) : P \in S^o\} \subset \mathcal C$  is an étale double cover $V \to U$. The composition
\[
V \to \mathcal{C}^{\text{n}} \xrightarrow{\phi^{\text{n}}} S
\]
coincides with the composition $V \to U \subset \Delta \cong S^o\subset S$. It is thus a generically finite map of degree $2$, and by \eqref{Eq:Stein} the field extension $K(S')/K(S)$ of degree $\geq 2$ embeds in the étale extension $K(V)/K(S)$ of degree $2$. Hence $K(S')=K(V)$. 

\vskip1mm

{\bf Linear disjointness.} Since the Mordell--Weil group of $\mathcal{E}$ is torsion-free \cite[Thm.~7.4]{SCSH}, the degree-$4$ extension $K(S)_2/K(S)$ given by multiplication by $[-2]$ has no non-trivial proper subextensions, as any of these would have degree $2$ and would give rise to a cyclic 2-isogeny on the elliptic surface $\mathcal E$, hence to a non-trivial $2$-torsion section. This immediately implies linear disjointness and concludes the proof. \qedhere
\end{proof}

\begin{proof}[Proof of Theorem \ref{thm:2}]
Consider the family of curves $\{C_P\}_{P\in S^o}$ in $S$, and take the corresponding family of inverse images $\{\rho^{-1}(C_P)\}_{P\in S^o}$ in $\mathcal E$. 

The theorem follows by applying Theorem \ref{thm:1} to $\mathcal E$ endowed with the fibration $\pi$ and the family $\{\rho^{-1}(C_P)\}_{P\in S^o}$. We verify the hypotheses below. 

\vskip1mm

{\bf Geometrically integral generic fiber.} The total space for the family $\{\rho^{-1}(C_P)\}_{P\in S^o}$ is the fiber product $\mathcal C':= \mathcal C \times_{\phi,S,\rho} \mathcal E$. The generic fiber of $\mathcal C' \to \mathcal E$ is the same as the generic fiber of $\mathcal C \to S$, and this is geometrically integral by Proposition \ref{prop:GIGF}.

\vskip1mm

{\bf Simple connectedness.} All the fibers of $\pi$ are geometrically integral (they are Weierstrass curves), hence the only irreducible divisors of $\mathcal E$ that are vertical for $\pi$ are precisely these fibers, i.e.\ the inverse images $\rho^{-1}(D)$ of the curves $D$ in the pencil $|{-K_S}|$. Proposition~\ref{prop:generic CR with fibers} implies that the general curve $C_P$ intersects all $D \in |{-K_S}|, D \neq \{w=0\}$ transversally , hence the general inverse image $\rho^{-1}(C_P)$ intersects $\rho^{-1}(D)$ transversally for all such $D$. It follows that $F$ is empty, so $S\setminus F=S$ is simply connected.

\vskip1mm

{\bf Zariski-density.} $\mathcal C(K)$ is dense in $\mathcal C$ by Proposition \ref{prop:points on T are dense} (essentially by Nijgh's work). Since $\mathcal C'$ is birational to $\mathcal C$, we see that $\mathcal C'(K)$ is dense in $\mathcal C'$. \qedhere
\end{proof}

To conclude, we prove Theorem~\ref{thm:3}.

\begin{proof}[Proof of Theorem~\ref{thm:3}]
    Let $U \subset \A_K^9$ be the non-empty Zariski-open whose $\bar K$-points are $9$-tuples $(a,b,\ldots,f_3)$ corresponding to a smooth $S$, and let $X \subset U \times \P(2,3,1,1)$ be the total space associated to the parametric family of Definition \ref{DefS}. Then the projection $\pi=pr_1:X \to U$ is smooth surjective and its fibers are del Pezzo surfaces of degree $1$. Note that $X$ is birational to the affine hypersurface $Y$ in $\A_K^{11}$ defined by the equation
    \[
    y^2  = x^3 + (af(z)+b) x + (cf(z)^2+df(z)+e),
    \]
    with $f(t) := f_3t^3 + f_2 t^2 + f_1 t + f_0$. The hypersurface $Y$ is rational via projection to any coordinate appearing linearly, e.g.\ $e$. 

    For $N \in \mathbb{Z}_{\geq 1}$, let $V_N \subset X$ be the non-empty Zariski-open corresponding to pairs $(\underline{a},P), \underline{a}=(a,\ldots,e,f_0,\ldots,f_3), P=[x_0:y_0:z_0:w_0]$ such that $w_0 \neq 0, 3z_0+2f_2w_0/f_3 \neq 0,$ the polynomial $f(t)-f(z_0/w_0)$ is separable, and $[N](\rho^{-1}(P)) \neq O$ on $\mathcal{E}_P$. Choose now $N$ to be divisible by the torsion exponents of all elliptic curves over $K$. This is a finite number by (the generalization of) Merel's theorem.

    For any $x \in V_N(K)$, the del Pezzo surface $\pi^{-1}(\pi(x))$ has the Hilbert property by Theorem \ref{thm:2}. Now, the morphism $\pi|_{V_N}:V_N \to U$ has geometrically integral generic fiber and thus sends non-thin sets to non-thin sets. In particular, since $V_N(K)$ is non-thin by Hilbert's Irreducibility Theorem (recall that $X$, and thus $V_N$, is rational), the set $\pi(V_N(K))$ is non-thin.

    Finally, for $\eta$ the generic point of $U$, we have the following properties for $X \to U$:
    \begin{enumerate}
        \item for every $u \in U$, we have $\operatorname{rk} \Pic X_u \geq \operatorname{rk} \Pic X_{\eta}$;
        \item the set of $u \in U$ such that $\operatorname{rk} \Pic X_u > \operatorname{rk} \Pic X_{\eta}$ is thin.
    \end{enumerate}
    The first property follows from injectivity of specialization on the Picard group (see e.g.\ \cite[\S5]{VA}). More precisely, in (2) we mean that there exists a finite collection of irreducible covers $\pi_i:Y_i \to U, i=1,\ldots,r$ of finite degrees $>1$ such that $\operatorname{rk} \Pic X_u > \operatorname{rk} \Pic X_{\eta}$ if and only if $u \in \operatorname{im} (Y_i(\kappa(u))\to U(\kappa(u)))$. This holds by applying Hilbert's Irreducibility Theorem to the family of degree-240 polynomials whose roots parametrize the lines in each surface in our family. The Galois orbits of lines correspond to the irreducible factors of these polynomials and determine the Picard rank. The Picard rank can only increase if an irreducible factor becomes reducible upon specialization, which occurs over a thin locus. Thus the Picard rank can only increase over a thin locus. By a result of Kloosterman \cite[Lem.~2.5,~Prop.~2.9]{Kloosterman}, reproving a result of Desjardins and Naskrecki \cite[Sec.~4.2]{DN}), the del Pezzo surface
    \[
    y^2=x^3+Az^6+Bw^6,
    \]
    where we take $A$ and $B$ to be transcendental variables over $K$, has Picard rank $1$. This surface is a member of the family $X \to U$: it is the specialization at the generic point $\xi$ of the copy of $\A^2_K$ in $\A^9_K$ defined by $f_3=A, e=B, c=1$, and where all the other coordinates are trivial. Then by (1) we get $1 \leq \operatorname{rk} \Pic X_{\eta} \leq  \operatorname{rk} \Pic X_\xi =1$, thus $\operatorname{rk} \Pic X_{\eta}=1$. Hence by (2) the set $T \subset U(K)$ corresponding to del Pezzo surfaces with Picard rank $\geq 2$ is thin. Then $\pi(V_N(K)) \setminus T$ furnishes the sought non-thin set of parameters.
\end{proof}


\bibliographystyle{plain}
\bibliography{bibliography}

\begin{thebibliography}{10}

\bibitem{AL}
J.~{Alessandr{\`\i}} and D.~{Loughran}.
\newblock {Fermat near misses and the integral Hilbert Property}.
\newblock {\em arXiv e-prints}, page arXiv:2511.17456, November 2025.

\bibitem{SP}
Various authors.
\newblock The {S}tacks {P}roject.
\newblock \url{https://stacks.math.columbia.edu}.

\bibitem{BW}
J.~W. Bruce and C.~T.~C. Wall.
\newblock On the classification of cubic surfaces.
\newblock {\em J. London Math. Soc. (2)}, 19(2):245--256, 1979.

\bibitem{COC0}
S.~Coccia.
\newblock The {H}ilbert property for integral points of affine smooth cubic
  surfaces.
\newblock {\em J. Number Theory}, 200:353--379, 2019.

\bibitem{COC}
S.~Coccia.
\newblock A {H}ilbert irreducibility theorem for integral points on del {P}ezzo
  surfaces.
\newblock {\em Int. Math. Res. Not. IMRN}, (6):5005--5049, 2024.

\bibitem{CTS}
J.-L. Colliot-Th\'{e}l\`ene and J.-J. Sansuc.
\newblock Principal homogeneous spaces under flasque tori: applications.
\newblock {\em J. Algebra}, 106(1):148--205, 1987.

\bibitem{CON}
B.~Conrad.
\newblock Chow's {$K/k$}-image and {$K/k$}-trace, and the {L}ang-{N}\'{e}ron
  theorem.
\newblock {\em Enseign. Math. (2)}, 52(1-2):37--108, 2006.

\bibitem{CZ}
P.~Corvaja and U.~Zannier.
\newblock On the {H}ilbert property and the fundamental group of algebraic
  varieties.
\newblock {\em Math. Z.}, 286(1-2):579--602, 2017.

\bibitem{DEM21}
J.~Demeio.
\newblock Elliptic fibrations and the {H}ilbert property.
\newblock {\em Int. Math. Res. Not. IMRN}, (13):10260--10277, 2021.

\bibitem{DS}
J.~Demeio and S.~Streeter.
\newblock Weak approximation for del {P}ezzo surfaces of low degree.
\newblock {\em Int. Math. Res. Not. IMRN}, (13):11549--11576, 2023.

\bibitem{DSW}
J.~Demeio, S.~Streeter, and R.~Winter.
\newblock Weak weak approximation and the {H}ilbert property for degree 2 del
  {P}ezzo surfaces.
\newblock {\em Proc. Lond. Math. Soc. (3)}, 128(5):Paper No. e12601, 28, 2024.

\bibitem{DN}
J.~Desjardins and B.~Naskr\k{e}cki.
\newblock Geometry of the del {P}ezzo surface {$y^2 = x^3 + A m^6 + B n^6$}.
\newblock {\em Ann. Inst. Fourier (Grenoble)}, 74(5):2231--2274, 2024.

\bibitem{DW}
J.~Desjardins and R.~Winter.
\newblock Density of rational points on a family of del {P}ezzo surfaces of
  degree one.
\newblock {\em Adv. Math.}, 405:108489, 2022.
\newblock With an appendix by Jean-Louis Colliot-Th\'{e}l\`ene (in French).

\bibitem{FW}
G.~Faltings, G.~W\"{u}stholz, F.~Grunewald, N.~Schappacher, and U.~Stuhler.
\newblock {\em Rational points}.
\newblock Aspects of Mathematics, E6. Friedr. Vieweg \& Sohn, Braunschweig,
  third edition, 1992.
\newblock Papers from the seminar held at the Max-Planck-Institut f\"{u}r
  Mathematik, Bonn/Wuppertal, 1983/1984, With an appendix by W\"{u}stholz.

\bibitem{FJ}
M.~Fried and M.~Jarden.
\newblock {\em Field arithmetic}, volume~11 of {\em Ergebnisse der Mathematik
  und ihrer Grenzgebiete. 3. Folge}.
\newblock Springer, Cham, [2023] \copyright 2023.
\newblock Fourth edition, Revised by Moshe Jarden.

\bibitem{GCM1}
D.~Gvirtz-Chen and G.~Mezzedimi.
\newblock A {H}ilbert irreducibility theorem for {E}nriques surfaces.
\newblock {\em Trans. Amer. Math. Soc.}, 376(6):3867--3890, 2023.

\bibitem{GCM2}
D.~{Gvirtz-Chen} and G.~{Mezzedimi}.
\newblock {Non-thin rational points for elliptic K3 surfaces}.
\newblock {\em arXiv e-prints}, page arXiv:2404.06844, April 2024.

\bibitem{HAR}
R.~Hartshorne.
\newblock {\em Algebraic geometry}.
\newblock Graduate Texts in Mathematics, No. 52. Springer-Verlag, New
  York-Heidelberg, 1977.

\bibitem{JAB}
E.~Jabara.
\newblock Rational points on some elliptic surfaces.
\newblock {\em Acta Arith.}, 153(1):93--108, 2012.

\bibitem{Kloosterman}
R.~Kloosterman.
\newblock Determining explicitly the {M}ordell-{W}eil group of certain rational
  elliptic surfaces, 2025.
\newblock Preprint, \url{https://arxiv.org/abs/2506.19423}.

\bibitem{LIU}
Q.~Liu.
\newblock {\em Algebraic geometry and arithmetic curves}, volume~6 of {\em
  Oxford Graduate Texts in Mathematics}.
\newblock Oxford University Press, Oxford, 2002.

\bibitem{M}
Yu.~I. Manin.
\newblock {\em Cubic forms}, volume~4 of {\em North-Holland Mathematical
  Library}.
\newblock North-Holland Publishing Co., Amsterdam, second edition, 1986.
\newblock Algebra, geometry, arithmetic, Translated from the Russian by M.
  Hazewinkel.

\bibitem{Matsumura}
H.~Matsumura.
\newblock {\em Commutative ring theory}, volume~8 of {\em Cambridge Studies in
  Advanced Mathematics}.
\newblock Cambridge University Press, Cambridge, second edition, 1989.

\bibitem{NIJ}
W.~Nijgh.
\newblock Rational points on a family of del {P}ezzo surfaces of degree one.
\newblock Master thesis, Universiteit Leiden,
  {\url{https://studenttheses.universiteitleiden.nl/handle/1887/3273800}},
  2022.

\bibitem{POO}
B.~Poonen.
\newblock {\em Rational points on varieties}, volume 186 of {\em Graduate
  Studies in Mathematics}.
\newblock American Mathematical Society, Providence, RI, 2017.

\bibitem{SvL}
C.~Salgado and R.~van Luijk.
\newblock Density of rational points on del {P}ezzo surfaces of degree one.
\newblock {\em Adv. Math.}, 261:154--199, 2014.

\bibitem{SCSH}
M.~Sch\"utt and T.~Shioda.
\newblock {\em Mordell-{W}eil lattices}, volume~70 of {\em Ergebnisse der
  Mathematik und ihrer Grenzgebiete. 3. Folge. A Series of Modern Surveys in
  Mathematics}.
\newblock Springer, Singapore, 2019.

\bibitem{SER}
J.-P. Serre.
\newblock {\em Topics in {G}alois theory}, volume~1 of {\em Research Notes in
  Mathematics}.
\newblock A K Peters, Ltd., Wellesley, MA, second edition, 2008.
\newblock With notes by Henri Darmon.

\bibitem{STR}
S.~Streeter.
\newblock Hilbert property for double conic bundles and del {P}ezzo varieties.
\newblock {\em Math. Res. Lett.}, 28(1):271--283, 2021.

\bibitem{SD}
P.~Swinnerton-Dyer.
\newblock Weak approximation and {$R$}-equivalence on cubic surfaces.
\newblock In {\em Rational points on algebraic varieties}, volume 199 of {\em
  Progr. Math.}, pages 357--404. Birkh\"auser, Basel, 2001.

\bibitem{ULA1}
M.~Ulas.
\newblock Rational points on certain elliptic surfaces.
\newblock {\em Acta Arith.}, 129(2):167--185, 2007.

\bibitem{ULA2}
M.~Ulas.
\newblock Rational points on certain del {P}ezzo surfaces of degree one.
\newblock {\em Glasg. Math. J.}, 50(3):557--564, 2008.

\bibitem{UT}
M.~Ulas and A.~Togb\'{e}.
\newblock On the {D}iophantine equation {$z^2=f(x)^2\pm f(y)^2$}.
\newblock {\em Publ. Math. Debrecen}, 76(1-2):183--201, 2010.

\bibitem{VA}
A.~V\'{a}rilly-Alvarado.
\newblock Weak approximation on del {P}ezzo surfaces of degree 1.
\newblock {\em Adv. Math.}, 219(6):2123--2145, 2008.

\end{thebibliography}
\end{document}